\documentclass[reqno]{amsart}
\usepackage{amscd}
\usepackage{epsfig}  
\usepackage[utf8]{inputenc}
\usepackage{latexsym,amssymb}
\usepackage{amsfonts}
\usepackage[T1]{fontenc}
\usepackage{amsthm}
\usepackage{amsmath,multicol,multirow,eso-pic}
\usepackage{graphicx}
\usepackage{color}
\usepackage{arcs}
\usepackage{float}
\usepackage{tikz}
\usepackage{tkz-tab}
\usepackage{array}
\usetikzlibrary{intersections,arrows,decorations.pathmorphing,decorations.pathreplacing,decorations.text,decorations.fractals,fit,backgrounds,positioning,petri ,calc,through,shapes.misc,chains,scopes,mindmap,calendar,trees,shadows,patterns,shapes,decorations,external,folding,shadings}

\usepackage{pgfplots}

\usepackage[all,knot]{xy}

\usepackage{mathtools}

\DeclarePairedDelimiter{\diagfences}{(}{)}
\newcommand{\diag}{\operatorname{diag}\diagfences}
\theoremstyle{plain}
\newtheorem{thm}{Theorem}[section]

\newtheorem{prop}[thm]{Proposition}
\theoremstyle{definition}
\newtheorem{defn}{Definition}[section]

\numberwithin{equation}{section}

\makeatletter
\newcounter{knot@strings}
\newif\ifknot@draftmode
\tikzoption{save knot path}{\tikz@addmode{\pgfsyssoftpath@getcurrentpath\knot@tmppath\expandafter\global\expandafter\let#1=\knot@tmppath}}
\tikzoption{use knot path}{\tikz@addmode{\expandafter\pgfsyssoftpath@setcurrentpath#1}}
\tikzset{
  knot/draft mode/.is if=knot@draftmode
}

\newcommand{\strand}[1][]{%
  \stepcounter{knot@strings}%
  \expandafter\def\csname knot@string@opts@\the\value{knot@strings}\endcsname{#1}%
\path[save knot path=\csname knot@string@\the\value{knot@strings}\endcsname]}
\newcommand{\crossing}[2]{%
\expandafter\def\csname knot@crossing@#1\endcsname{#2}}

\makeatother

\begin{document}

\title[Irreducible representations of the Braid Group]{IRREDUCIBLE REPRESENTATIONS OF THE BRAID GROUP $B_3$ in dimension 6}

\author{Taher I. Mayassi \and Mohammad N. Abdulrahim }

\address{Taher I. Mayassi\\
         Department of Mathematics and Computer Science\\
         Beirut Arab University\\
        P.O. Box 11-5020, Beirut, Lebanon}
\email{tim187@student.bau.edu.lb}

\address{Mohammad N. Abdulrahim\\
         Department of Mathematics and Computer Science\\
         Beirut Arab University\\
         P.O. Box 11-5020, Beirut, Lebanon}
\email{mna@bau.edu.lb}

\begin{abstract}
We use $q$-Pascal's triangle to define a family of representations of dimension 6 of the braid group $B_3$ on three strings. Then we give a necessary and sufficient condition for these representations to be irreducible.

\end{abstract}

\maketitle

\medskip

\renewcommand{\thefootnote}{}
\vskip 0.1in 

\section{Introduction}
\bigskip
Braid groups have an important role in many branches of mathematics like Knot Theory and Cryptography. In this work, we study the irreducibility of representations of the Braid group $B_3$ of dimension 6.
In \cite{Al}, a family of representations of $B_3$ of dimension $n+1$ is constructed using $q$-deformed Pascal's triangle.
This family of representations of $B_3$ is a generalization of the representations given by Humphries \cite{Hu} as well as the representations given by I. Tuba and H. Wenzl \cite{TuW}.
For more details, see [1,Theorem 3] and \cite{AR}.
Kosyak mentioned in \cite{Ko} that the irreducibility of the representations constructed by $q$-Pascal's
triangle is still an open problem for dimensions $\geq6$, although some sufficient conditions are given in \cite{Al}.
In our work, we consider these representations and we determine a necessary and sufficient condition for the irreducibility in the case the dimension is precisely 6.\\
In section 2, we use some notations that help us define a family of representations of $B_3$ by $q$-Pascal's triangle (see Theorem 2.1).
In section 3, we specialize the representations defined by Theorem 2.1 to $n=5$, that is of dimension 6, by taking some specific values of some parameters.
We get a subfamily of representations of $B_3$.
Proposition 3.1 proves that these representations have no invariant subspaces of dimension 1.
However, Proposition 3.2, Proposition 3.3 and Proposition 3.4 state necessary and sufficient conditions for the non-existence of invariant subspaces of dimensions 2, 3 and 4 respectively.
Proposition 3.5 gives a sufficient condition for these representations to have no invariant  subspaces of dimension 5.
Our main result is Theorem 3.6, which determines a necessary and sufficient condition for the irreducibility of this family of representations of $B_3$.
In section 4, we consider the cases where the representations are reducible.
Then, we reduce one of these reducible representations to a sub-representation of dimension 4 and we prove that this sub-representation is irreducible (Theorem 4.1).
\section{Notations, Definitions and Basic Theorems}
\bigskip
\begin{defn}\cite{Bi}
The braid group on $n$ strings,  $B_n$, is the abstract group with $n-1$ generators $\sigma_1,\sigma_2,\cdots,\sigma_{n-1}$ satisfying the following relations
$$\sigma_i\sigma_j=\sigma_j\sigma_i\text{ for }|i-j|>1\text{ and }\sigma_i\sigma_{i+1}\sigma_i=\sigma_{i+1}\sigma_{i}\sigma_{i+1}\\
\text{ for }i=1,\cdots,n-2.$$
\end{defn}

\noindent In order to use the $q$-Pascal's triangle, we need the following notations.\\ 

\noindent
\textbf{Notations.} \cite{Al}
For every $(n\times n)$-matrix $M=(m_{ij})$, we set the matrices  $M^\sharp=(m_{ij}^\sharp)$ and $M^s=(m_{ij}^s)$ where $m_{ij}^\sharp=m_{n-i,n-j}$ and
$m_{ij}^s=m_{n-j,n-i}$.\\\\
For $q\in\mathbb{C}\setminus\{0\}$, $n\in\mathbb{N}$, and for all integers $j$ and $r$ such that $j>0$ and $r\geqslant0$ we define the following terms.
$$\begin{array}{l}
(j)_q=1+q+\cdots+q^{j-1},\\\\
(j)!_q=(1)_q(2)_q\cdots(j)_q\text{ and }(0)!_q=1,\\\\
\displaystyle{n \choose r}_q=\dfrac{(n)!_q}{(r)!_q(n-r)!_q},\text{ for all integers }r\text{ and }n\text{ such that } 0\leqslant r\leqslant n,\\\\
q_{r}=q^{\dfrac{(r-1)r}{2}}.
\end{array}
$$

\noindent $(j)_q$ and ${n \choose r}_q$ are called $q$-natural numbers and $q$-binomial coefficients respectively.\\

\begin{defn}(\cite{Al},\cite{Ko})
Let $n$ be a non-negative integer. For all non-zero complex numbers $q$, $\lambda_0,\;\lambda_1,\cdots,\; \lambda_n$, consider the matrices
$$D_n(q)=\diag{q_{r}}_{r=0}^n,\;\; \Lambda_n=\diag{\lambda_0,\lambda_1,\cdots,\lambda_n}\text{ and }A_n(q)=\left(a_{km}\right)_{0\leqslant k,m\leqslant n},$$ 
where $a_{km}={n-k\choose n-m}_q=\frac{(n-k)!_q}{(n-m)!_q(m-k)!_q},$ for $k\leqslant m$ and $a_{km}=0$ for $k>m$.\\\\
We define the following family of $(n+1)\times(n+1)$-matrices 
$$\sigma_1^{\Lambda_n}(q,n)=A_n(q)D_n^\sharp(q)\Lambda_n \text{ and } \sigma_2^{\Lambda_n}(q,n)=\Lambda_n^\sharp D_n(q)\left(\left(A_n\left(q^{-1}\right)\right)^{-1}\right)^\sharp.$$
\end{defn}
Using the definitions and notations above we state the following theorem.
\begin{thm}\cite{Al}
The mapping $B_3\to GL(n+1,\mathbb{C})$ defined by
$$\sigma_1\mapsto \sigma_1^{\Lambda_n}(q,n) \text{ and } \sigma_2\mapsto \sigma_2^{\Lambda_n}(q,n)$$
is a representation of dimension $n+1$ of the braid group $B_3$ provided that $\lambda_i\lambda_{n-i}=c$ for $0\leqslant i\leqslant n$, where $c$ is a constant non-zero complex number.
\end{thm}
\begin{defn}
A representation is called \textit{subspace irreducible} or \textit{irreducible}, if there are no non-trivial invariant subspaces for all operators of the representation.
A representation is called \textit{operator irreducible}, if there are no non-trivial bounded operators commuting with all operators of the representation.

\end{defn}

For the next theorem, we need to introduce the following operators. For $n,r\in\mathbb{N}$ such that $0\leqslant r\leqslant n$, and for $\lambda=(\lambda_0,\dots,\lambda_n)\in\mathbb{C}^{n+1}$ and $q\in\mathbb{C}$, we define
$$F_{r,n}(q,\lambda)=\exp_{(q)}\left(\sum_{k=0}^{n-1}(k+1)_{q}E_{k\;k+1}\right)-q_{n-r}\lambda_r\left(D_n(q)\Lambda_n^\sharp\right)^{-1},$$
where  $E_{km}$ is a matrix with 1 in the $k,m$ entry and zeros elsewhere ($k,m\in\mathbb{Z})$, and $\exp_{(q)}(X)=\sum_{m=0}^{\infty}(X^m/(m)!_q)$.
For the $(n+1)\times(n+1)$-matrix $C$ over $\mathbb{C}$ and for $0\leqslant i_0<i_1<\cdots<i_r\leqslant n$, $0\leqslant j_0<j_1<\cdots<j_r\leqslant n$, we denote the minors of $C$ with $i_1,i_2,\dots,i_r$ rows and $j_1,j_2,\dots,j_r$ columns by
$$M_{j_1j_2\dots j_r}^{i_1i_2\dots i_r}(C).$$
\begin{thm}\cite{Al}
The representation of the group $B_3$ defined in Theorem 2.1 has the following properties:
\begin{enumerate}
\item for $q=1$, $\Lambda_n=I_{n+1}$ (the identity matrix), it is subspace irreducible in arbitrary dimension $n\in\mathbb{N}$;
\item  for $q=1$, $\Lambda_n=\diag{\lambda_0,\lambda_1,\cdots,\lambda_n}\neq I_{n+1}$, it is operator irreducible if and only if for
any $0\leqslant r \leqslant [\frac{n}{2}]$, there exists 
$0\leqslant i_0<i_1<\dots<i_r\leqslant n$ such that
$$M_{r+1\:r+2\dots n}^{i_0i_1\dots i_{n-r-1}}\left(F_{r,n}^{s}(q,\lambda)\right)\neq0;$$
\item for $q\neq1$, $\Lambda_n=I_n$, it is subspace irreducible if and only if $(n)_q\neq0$.\\
The representation has $[\frac{n+1}{2}]+1$ free parameters.

\end{enumerate}
\end{thm}
\begin{thm}\cite{Al}
All representations of the braid group $B_3$ of dimension $\leqslant5$ are the representations defined in Theorem 2.1.
\end{thm}
Note that the irreducibility of the representation defined in Theorem 2.1 that are of dimension $\leqslant5$ is discussed in \cite{Tu} and \cite{TuW}.
Also, Theorem 2.1 gives a family of the representations of $B_3$ that are of dimension $\geqslant6$ ($n\geqslant5$).
The subject of the irreducibility of these representations is still under study and research.
In this work, we study the irreducibility of some of these representations of dimension 6.\\

Suppose, in what follows, that $n=5$. Then, the matrix $A_5(q)$ is given by  
$$A_5(q)=\begin{pmatrix}
1&(5)_q&(1+q^2)(5)_q&(1+q^2)(5)_q&(5)_q&1\\
0&1&(1+q)(1+q^2)&(1+q^2)(3)_q&(1+q)(1+q^2)&1\\
0&0&1&(3)_q&(3)_q&1\\
0&0&0&1&1+q&1\\
0&0&0&0&1&1\\
0&0&0&0&0&1\\
\end{pmatrix}$$
and the representation of $B_3$ defined in Theorem 2.1 is of dimension 6. Moreover, the matrices representing the generators $\sigma_1$ and $\sigma_2$ of $B_3$ are given by

$$\sigma_1\mapsto \begin{pmatrix}
\lambda_0q^{10}&\lambda_1q^6(5)_q&\lambda_2q^3(1+q^2)(5)_q&\lambda_3q(1+q^2)(5)_q&\lambda_4(5)_q&\lambda_5\\
0&\lambda_1q^6&\lambda_2q^3(1+q)(1+q^2)&\lambda_3q(1+q^2)(3)_q&\lambda_4(1+q)(1+q^2)&\lambda_5\\
0&0&\lambda_2q^3&\lambda_3q(3)_q&\lambda_4(3)_q&\lambda_5\\
0&0&0&\lambda_3q&\lambda_4(1+q)&\lambda_5\\
0&0&0&0&\lambda_4&\lambda_5\\
0&0&0&0&0&\lambda_5\\
\end{pmatrix}$$
and 
$$\sigma_2\mapsto \begin{pmatrix}
\lambda_5&0&0&0&0&0\\

-\lambda_4&\lambda_4&0&0&0&0\\

\lambda_3&-\lambda_3(1+q)&\lambda_3q&0&0&0\\

-\lambda_2&\lambda_2(3)_q&-\lambda_2q(3)_q&\lambda_2q^3&0&0\\

\lambda_1&-\lambda_1(4)_q&\lambda_1q(1+q^2)(3)_q&-\lambda_1q^3(4)_q&\lambda_1q^6&0\\

-\lambda_0&\lambda_0(5)_q&-\lambda_0q(1+q^2)(5)_q&\lambda_0q^3(1+q^2)(5)_q&-\lambda_0q^6(5)_q&\lambda_0q^{10}\\
\end{pmatrix}.$$
\bigskip
\section{Irreducibility of Representations of $B_3$ of dimension 6}
\bigskip
In this section, let $q$ be a primitive third root of unity ($q^3=1$ and $q\neq1$).
By taking $c=1$, $\lambda_0=1$ and $\lambda_2=q^2$, we get $\lambda_3=\frac{1}{q^2}=q$, $\lambda_4=\lambda_1^{-1}$ and $\lambda_5=\frac{c}{\lambda_0}=1$. Under these conditions and for $n=5$, we substitute these values in the matrices above to get the following definition. 
\begin{defn}
Let $\rho:B_3\to GL(6,\mathbb{C})$ be the family of representations of $B_3$ of dimension 6 that is defined by
$$\sigma_1\mapsto
\begin{pmatrix}
q&-q^2\lambda_1&q^2&q^2&-q^2\lambda_1^{-1}&1\\
0&\lambda_1 &q^2&0&\lambda_1^{-1}&1\\
0&0&q^2&0&0&1\\
0&0&0&q^2&-q^2\lambda_1^{-1}&1\\
0&0&0&0&\lambda_1^{-1}&1\\
0&0&0&0&0&1\\
\end{pmatrix}$$
and 
$$\sigma_2\mapsto
\begin{pmatrix}
1&0 &0&0&0&0\\
-\lambda_1^{-1}&\lambda_1^{-1} &0&0&0&0\\
q&1&q^2&0&0&0\\
-q^2&0&0&q^2&0&0\\
\lambda_1&-\lambda_1&0&-\lambda_1&\lambda_1&0\\
-1&-q^2&-q&1&q^2&q\\
\end{pmatrix}.$$
\end{defn}
Note that $\rho(\sigma_1)$ and $\rho(\sigma_2)$ have the same eigenvalues, which are $q,\lambda_1,q^2$ (of multiplicity 2), $\lambda_1^{-1}$ and $1$.
The corresponding eigenvectors of  $\rho(\sigma_1)$ are
$$u_1=\begin{pmatrix}
1\\
0\\
0\\
0\\
0\\
0\\
\end{pmatrix},\; 
u_2=\begin{pmatrix}
-\lambda_1(1+q)\\
-\lambda_1+q\\
0\\
0\\
0\\
0\\
\end{pmatrix},\;
u_3=\begin{pmatrix}
\lambda_1+q\\
q-1\\
(\lambda_1+q^2)(-q+q^2)\\
0\\
0\\
0\\
\end{pmatrix},\;
u_4=\begin{pmatrix}
-1\\
0\\
0\\
q^2-1\\
0\\
0\\
\end{pmatrix},$$

$$u_5=\begin{pmatrix}
q-\lambda_1^3\\
(1-\lambda_1q)(\lambda_1q-q^2)\\
0\\
-q(1-\lambda_1^2)(-1+\lambda_1q)\\
(-1+\lambda_1^2)(-1+\lambda_1q)(\lambda_1q-q^2)\\
0\\
\end{pmatrix} \text{ and }
u_6=\begin{pmatrix}
\lambda_1q^2-2\lambda_1+3+3\lambda_1^2+\lambda_1q\\
3q(\lambda_1q-1)\\
-3q^2(1-\lambda_1)^2\\
-3(-1+\lambda_1)(1+\lambda_1q^2)\\
3q(1-q)\lambda_1(-1+\lambda_1)\\
3q(1-q)(-1+\lambda_1)^2\\
\end{pmatrix}.$$
Assume that $\lambda_1\not\in\{-1,1,q,q^2\}$. Then, the vectors $u_i (i=1,2,\cdots,6)$ are linearly independent and the transition matrix $P=(u_1\;u_2\;u_3\;u_4\;u_5\;u_6)$ is invertible. Conjugating the representation by $P$, we get an equivalent representation.
$$\rho(\sigma_1)\mapsto X=\begin{pmatrix}
q&0&0&0&0&0\\
0&\lambda_1&0&0&0&0\\
0&0&q^2&0&0&0\\
0&0&0&q^2&0&0\\
0&0&0&0&\lambda_1^{-1}&0\\
0&0&0&0&0&1\\
\end{pmatrix}$$

and $\rho(\sigma_2)\mapsto Y= P^{-1}\rho(\sigma_2)P=\left(K_1\;K_2\;K_3\;K_4\;K_5\;K_6\right)$, where

$$K_1=\begin{pmatrix}
\frac{\lambda_1q}{(q^2-1)(-\lambda_1+q)(-1+\lambda_1q)}\\
\frac{\lambda_1^3-q^2}{\lambda_1(\lambda_1-1)^2(\lambda_1+1)(\lambda_1-q)(\lambda_1-q^2)}\\
-\frac{1}{3(-\lambda_1+q^2)}\\
-\frac{q^2(\lambda_1+q^2)}{3(-\lambda_1+q)}\\
\frac{\lambda_1^2}{(\lambda_1-1)^2(\lambda_1+1)(\lambda_1-q)q(\lambda_1q-1)}\\
\frac{1}{3(\lambda_1-1)^2(q^2-q)}
\end{pmatrix},\;\;
K_2=\begin{pmatrix}
\frac{q[2+q-(1+2q)\lambda_1^2]}{3(\lambda_1-q)(-1+\lambda_1q)}\\
\frac{1-\lambda_1^3q^2}{\lambda_1(\lambda_1-1)^2(\lambda_1+1)(\lambda_1-q)(\lambda_1-q^2)}\\
-\frac{1}{3(-\lambda_1+q^2)}\\
\frac{1+\lambda_1}{3(\lambda_1-q)}\\
-\frac{\lambda_1(\lambda_1^2+q)}{(\lambda_1-1)^2(\lambda_1+1)(\lambda_1-q)(\lambda_1q-1)}\\
\frac{1}{3(\lambda_1-1)^2(q^2-q)}
\end{pmatrix},
$$

$$K_3=\begin{pmatrix}
\frac{(1+\lambda_1)(-3q^2+(1-q^2)\lambda_1+3\lambda_1^2)}{-3(\lambda_1-q)(-1+\lambda_1q)}\\
\frac{q^2(-1-\lambda_1+(q-q^2)\lambda_1^2+\lambda_1^3+\lambda_1^4)}{\lambda_1(\lambda_1-1)^2(\lambda_1+1)(\lambda_1-q)(\lambda_1-q^2)}\\
-\frac{\lambda_1+q}{3q(-\lambda_1+q^2)}\\
-\frac{q(2+\lambda_1+2\lambda_1^2)}{3(\lambda_1-q)}\\
\frac{q\lambda_1^2(1+q\lambda_1))}{(\lambda_1-1)^2(\lambda_1+1)(\lambda_1-q)(\lambda_1q-1)}\\
\frac{q(q+\lambda_1)}{3(\lambda_1-1)^2(q-1)}
\end{pmatrix},\;\;
K_4=\begin{pmatrix}
\frac{-3+2(q-1)\lambda_1+3q\lambda_1^2}{-3(\lambda_1-q)(-1+\lambda_1q)}\\
\frac{q^2+(q-q^2)\lambda_1+(q-q^2)\lambda_1^2-q\lambda_1^3}{\lambda_1(\lambda_1-1)^2(\lambda_1+1)(\lambda_1-q)(\lambda_1-q^2)}\\
\frac{-2}{3q(-\lambda_1+q^2)}\\
\frac{q(\lambda_1+q^2)}{3(-\lambda_1+q)}\\
\frac{-q\lambda_1^2}{(\lambda_1-1)^2(\lambda_1+1)(\lambda_1-q)(\lambda_1q-1)}\\
-\frac{q}{3(\lambda_1-1)^2(q-1)}
\end{pmatrix},
$$

$$K_5=\begin{pmatrix}
\frac{\lambda_1(-2q-1+(2+q)\lambda_1^2+(2+q^2)\lambda_1^3+(2q+1)\lambda_1^5)}{3(\lambda_1-q)(-1+\lambda_1q)}\\
\frac{q^2+\lambda_1^2+q^2\lambda_1^3+\lambda_1^5+q^2\lambda_1^6+\lambda_1^8}{\lambda_1(\lambda_1-1)^2(\lambda_1+1)(\lambda_1-q)(\lambda_1-q^2)}\\
\frac{\lambda_1(1+q+(2+q)\lambda_1-(1+2q)\lambda_1^2+q^2\lambda_1^3)}{3q^2(-\lambda_1+q^2)}\\
-\frac{q^2\lambda_1(1+q^2-\lambda_1-(2+q)\lambda_1^2+q^2\lambda_1^3+q\lambda_1^4)}{3(\lambda_1-q)}\\
\frac{\lambda_1^3(-1+q\lambda_1^3)}{(\lambda_1-1)^2(\lambda_1+1)(\lambda_1-q)(\lambda_1q-1)}\\
-\frac{\lambda_1(-q+\lambda_1^3))}{3(\lambda_1-1)^2(q-1)}
\end{pmatrix},\;\;
K_6=\begin{pmatrix}
\frac{(q^2-1)(1-\lambda_1^2+\lambda_1^4)}{(\lambda_1-q)(-1+\lambda_1q)}\\
\frac{3q+3q^2\lambda_1^2-3q^2\lambda_1^3-3\lambda_1^5}{\lambda_1(\lambda_1-1)^2(\lambda_1+1)(\lambda_1-q)(\lambda_1-q^2)}\\
\frac{2+q+2(q^2-q)\lambda_1+(-1+q)\lambda_1^2}{(1-q)(-\lambda_1+q^2)}\\
\frac{q-q^2-3\lambda_1+3\lambda_1^2+(1+2q)\lambda_1^3}{(1-q)(\lambda_1-q)}\\
\frac{-3q^2\lambda_1^2-3q\lambda_1^4}{(\lambda_1-1)^2(\lambda_1+1)(\lambda_1-q)(\lambda_1q-1)}\\
\frac{-2-q+(1-q^2)\lambda_1+(-1+q^2)\lambda_1^2}{3(\lambda_1-1)^2}
\end{pmatrix}.
$$\\

\begin{prop}
The representation $\rho$ has no invariant subspaces of dimension 1.
\end{prop}
\begin{proof}
The possible subspaces of dimension 1 that are invariant under $X$ are $\langle e_i\rangle$ for $i=1,2,3,4,5,6$, and $\langle\alpha e_3+e_4\rangle$ for $\alpha\in\mathbb{C}$. Here $e_i$ are the standard unit vectors in $\mathbb{C}^6$, and are considered as column vectors.\\
It is clear to see that
$Y(e_i)\not\in\langle e_i\rangle$ for $i=1,2,3,4$.
Assume that $Y(e_5)=K_5\in\langle e_5\rangle$. Then all components, except the fifth, of $K_5$ are zeros.
In particular the third and sixth components of $K_5$ are zeros. So,  $\lambda_1^3=q$ and $1+q+(2+q)\lambda_1-(1+2q)\lambda_1^2+q^2\lambda_1^3=0$. Then, by direct computation, we get $\lambda_1=q$ or $q^2$ and $\lambda_1^3=q$ which is impossible as $q\neq1$. So, $Y(e_5)\not\in\langle e_5\rangle$.
Therefore, $\langle e_5\rangle$ is not invariant under $Y$.\\
Suppose that $Y(e_6)\in\langle e_6\rangle$. Then the $5^{\text{th}}$ and $2^{\text{nd}}$ components of $K_6=Y(e_6)$ are zeros.
Then $\lambda_1^2=-q$ and $q+q^2\lambda_1^2-q^2\lambda_1^3-\lambda_1^5=0$.
By direct computation, we get $\lambda_1=-q$.
 Therefore, $q=-1$, contradiction. So, $Y(e_6)\not\in\langle e_6\rangle$ and $\langle e_6\rangle$ is not invariant under $Y$.\\
It remains to show that any subspace of the form $\langle\alpha e_3+e_4\rangle$, where $\alpha $ is a non-zero complex number, is not invariant under $Y$. 
Note that $Y(\alpha e_3+e_4)=\alpha K_3+K_4$. If $\alpha K_3+K_4\in\langle\alpha e_3+e_4\rangle$ then
the fifth and sixth components of $\alpha K_3+K_4$ are zeros. So, $\alpha q\lambda_1^2(1+q\lambda_1)+q\lambda_1^2=0$ and $\alpha q(q+1)-q=0$. This implies that $\alpha =-q$ and $\lambda_1=q-q^2$. Substitute the obtained values of $\alpha $ and $\lambda_1$ in the numerator of the $2^{\text{nd}}$ component of $ \alpha K_3+K_4$ to get
$$
-(q-q^2)^3-(q-q^2)^4+q^2+(q-q^2)^2-q(q-q^2)^3=2(q-8)\neq 0
\text{ when }q^3=1, q\neq1.$$
Therefore, the second component of $ \alpha K_3+K_4$  is not zero, contradiction.
\end{proof}
\begin{prop}
The representation $\rho$ has no invariant subspaces of dimension 2 if and only if $\lambda_1^3\neq q$.
\end{prop}
\begin{proof}
The possible subspaces of dimension 2 that are invariant under $X$ are: $S_{ij}=\langle e_i,e_j\rangle$, and $S^{\alpha}_k=\langle e_k,\alpha e_3+e_4\rangle$ for $\alpha\in\mathbb{C}$, $1\leqslant i<j\leqslant6$ and $k=1,2,5,6$.\\
We can easily see that $Y(e_1)=K_1\not\in S_{1i}$ for all $i=2,3,4,5,6$. So, the subspaces $S_{1i}$ $(i=2,3,4,5,6)$ are not invariant under $Y$.\\
Also $Y(e_2)=K_2\not\in S_{2i}$ for $i=3,4,5,6$ since the third and sixth components of $K_2$ are not zeros. Thus, the subspaces $S_{2i}$ $(i=3,4,5,6)$ are not invariant under $Y$.\\
The fourth, fifth and sixth components of $Y(e_3)=K_3$ cannot be zeros at the same time for all values of $\lambda_1$. So, $Y(e_3)=K_3\not\in S_{3i}$ for $i=4,5,6$. So, $S_{3i}$ is not invariant under $Y$ for $i=4,5,6$.\\
$Y(e_4)\not\in S_{4i}$ for $i=5,6$ because the third component of $Y(e_4)=K_4$ is not zero. So, $S_{4i}$ is not invariant under $Y$ for $i=5,6$.\\
If the third  and fourth components of $Y(e_5)=K_5$ are zeros and since $\lambda_1\neq0$ we then have
$$\left\{\begin{array}{c}
1+q+(2+q)\lambda_1-(1+2q)\lambda_1^2+q^2\lambda_1^3=0\\
1+q^2-\lambda_1-(2+q)\lambda_1^2 +q^2\lambda_1^3 +q \lambda_1^4 =0\end{array}\right.$$

By using Mathematica, we show that there is no complex solution in terms of $q$ satisfying this system of equations.
So $Y(e_3)\not\in S_{56}$. Therefore, $ S_{56}$ is not invariant under $Y$.\\
It remains to discuss the subspaces  $S^\alpha_{k}$ for $k=1,2,5,6$ and $\alpha\in\mathbb{C}$. Since the sixth components of $Y(e_1)$ and $Y(e_2)$ are not zeros, it follows that $Y(e_1)\not\in\langle\alpha e_3+e_4,e_1\rangle$ and  $Y(e_2)\not\in\langle\alpha e_3+e_4,e_2\rangle$ for all $\alpha\in\mathbb{C}$. Therefore, $S^\alpha_{1}$ and $S^\alpha_{2}$ are not invariant under $Y$.
Now, if $Y(e_6)=K_6\in\langle\alpha e_3+e_4,e_6\rangle$ then the second and fifth components of $K_6$ are zeros.
This yields to the following equations:
$$
\left\{\begin{array}{c} q+q^2\lambda_1^2-q^2\lambda_1^3-\lambda_1^5=0\\
-3q\lambda_1^2(q+\lambda_1^2)=0
\end{array}
\right..$$
Using that fact that $q$ is a primitive third root of unity, we get  $\lambda_1=\lambda_1^2=-q$, that is $q=\lambda_1=0 \text{ or } 1$, a contradiction.
Hence, $Y(e_6)\not\in S^\alpha_{6}$ for all $\alpha\in\mathbb{C}$. Therefore, $S^\alpha_{6}$ is not invariant under $Y$.\\   
Finally, if a subspace $S^\alpha_{5}$ is invariant under $Y$ then, $Y(e_5)=K_5\in\langle S^\alpha_{5}\rangle$. Then, $\lambda_1^3=q$. Conversely, if $\lambda_1^3=q$ 
then, by direct computation and using Mathematica, we show that 
$Y(e_5)=K_5\in S^\alpha_{5}$  and
$Y(\alpha e_3+e_4)\in S^\alpha_{5}$ for $\alpha=(-1-\lambda_1)(1+q\lambda_1)$.
Therefore, $\langle\alpha e_3+e_4,e_5\rangle$ is invariant under $Y$ if and only if $\lambda_1^3=q$. The invariant subspaces corresponding to these values of $\lambda_1$ are of the form
$$\langle\alpha e_3+e_4,e_5\rangle \text{ where, }\alpha=(-1-\lambda_1)(1+q\lambda_1).$$
\end{proof}

\begin{prop}
The representation $\rho$ has no invariant subspaces of dimension 3 if and only if $\lambda_1^2\neq-q$, $-q^2$. 
\end{prop}
\begin{proof}
The subspaces of dimension 3 that are invariant under $X$ are $\langle e_i,e_j,e_k\rangle$ and $\langle e_s,\alpha e_3+e_4,e_t\rangle$ for $\alpha\in\mathbb{C}$, $1\leqslant i<j<k\leqslant6$ and $s,t\in\{1,2,5,6\}$ with $s<t$.\\
Since the third, fifth and sixth components of $Y(e_1)$ are not zeros, it follows that all the subspaces of the form $\langle e_1,e_j,e_k\rangle$ together with the subspaces of the form $\langle e_1,\alpha e_3+e_4,e_t\rangle$ are not invariant under $Y$ for $1<j<k\leqslant 6$ and $t=2,5,6$.\\
The third and sixth components of $K_2=Y(e_2)$ are not zeros. So, $Y(e_2)\not\in\langle e_2,e_j,e_k\rangle$ and $Y(e_2)\not\in\langle e_2,\alpha e_3+e_4,e_5\rangle$ for all $2<j<k\leqslant 6$ such that $\{j,k\}\neq\{3,6\}$ and for all $\alpha\in\mathbb{C}$. Then the subspaces of the form $\langle e_2,e_j,e_k\rangle$ and $\langle e_2,\alpha e_3+e_4,e_5\rangle$ are not invariant under $Y$ for all $2<j<k\leqslant 6$ such that $\{j,k\}\neq\{3,6\}$.\\
Assume that the subspace $S=\langle e_2,e_3,e_6\rangle$ is invariant under $Y$ then, $Y(e_3)\in S$. So, $\lambda_1=-q$ (as the sixth component of $Y(e_3)$ is zero). Substitute the value of $\lambda_1$ in the first component of $Y(e_3)$ to get $\frac{(q-1)(-q^2+1)}{6q}\neq0$, contradiction. So, $S$ is not invariant under $Y$.\\
Consider the subspace $S^\alpha=\langle e_2,\alpha e_3+e_4,e_6\rangle$, where $\alpha\in\mathbb{C}$.
 Suppose $S^\alpha$ is invariant under $Y$ then, $Y(e_2)\in S^\alpha$. So, the fifth component of $K_2$  is zero. Thus, $\lambda_1^2=-q$.
Conversely, assume that $\lambda_1^2=-q$. Then, by direct computation and using Mathematica, we show that $Y(e_r)\in S^\alpha$ for $r=2,6$ and $Y(\alpha e_3+e_4)\in S^\alpha$
and in this case $\alpha=\frac{1}{2}\pm\frac{1}{2}i$. Therefore, the subspace $\langle e_2,\alpha e_3+e_4,e_6\rangle$ is invariant under $Y$ if and only if $\lambda_1^2=-q$.\\
Since the sixth and fifth components of  $Y(e_4)$ are not zeros, it follows that $Y(e_4)\not\in\langle e_3,e_4,e_5\rangle$ and $Y(e_4)\not\in\langle e_3,e_4,e_6\rangle$. Thus, the subspaces $\langle  e_3,e_4,e_5\rangle$ and $\langle  e_3,e_4,e_6\rangle$ are not invariant under $Y$.\\
Assume that $Y(e_3)\in\langle e_3,e_5,e_6\rangle$. Then the fourth component of $K_3$ is zero. So, $\lambda_1=\frac{-1\pm i\sqrt{15}}{4}$. But for this value of $\lambda_1$ and by direct calculation,  the first component of $K_3$ is $-\frac{9}{8}\pm\frac{21\sqrt{5}}{16}+i\left(\pm\frac{9\sqrt{3}}{16}\pm\frac{3\sqrt{15}}{8}\right)\neq0$. This is a contradiction. Thus the subspace $\langle e_3,e_5,e_6\rangle$ is not invariant under $Y$.\\
Since the third component of $K_4$ is not zero it follows that $Y(e_4)\not\in\langle e_4,e_5,e_6\rangle$. So, the subspace $\langle e_4,e_5,e_6\rangle$ is not invariant under $Y$.\\
Consider the subspace $V=\langle e_5,\alpha e_3+e_4,e_6\rangle$. By using Mathematica, we show that $Y(e_6)\in V$ if and only if $\lambda_1^2=-q^2$. 
Moreover,  $$\alpha=\left\{\begin{array}{c}\frac{2}{11}(4+i-3\lambda_1-\lambda_1^2)\;\;\text{ for } \lambda_1=iq \\
\frac{2}{11}(4-i-3\lambda_1-\lambda_1^2)\;\;\text{ for } \lambda_1=-iq
\end{array}
\right.$$
Also, we show that if $\lambda_1^2=-q^2$ then $Y(e_5)\in V$ and $Y(\alpha e_3+e_4)\in V$ for
$$\alpha=\left\{\begin{array}{c}\frac{2}{11}(4+i-3\lambda_1-\lambda_1^2)\;\;\text{ for } \lambda_1=iq \\
\frac{2}{11}(4-i-3\lambda_1-\lambda_1^2)\;\;\text{ for } \lambda_1=-iq
\end{array}
\right..$$
Therefore, $\langle e_5,\alpha e_3+e_4,e_6\rangle$ is invariant under $Y$ if and only if $\lambda_1^2=-q^2$.
\end{proof}
\begin{prop}
The representation $\rho$ has no invariant subspaces of dimension 4 if and only if $\lambda_1^3\neq q^2$.
\end{prop}
\begin{proof}
The subspaces of dimension 4 that are invariant under $X$ are $\langle e_i,e_j,e_k,e_r\rangle$ and $\langle \alpha e_3+e_4,e_s,e_t,e_h\rangle$ for $\alpha\in\mathbb{C}$, $1\leqslant i<j<k<r\leqslant6$ and $s,t,h\in\{1,2,5,6\}$ with $s<t<h$.\\
Since the fifth and sixth components of $Y(e_1)$ are not zeros, it follows that the subspaces of the form $\langle e_1,e_2,e_3,e_i\rangle$, $\langle e_1,e_2,e_4,e_j\rangle$, $\langle e_1,e_3,e_4,e_j\rangle$, $\langle \alpha e_3+e_4, e_1,e_2,e_5\rangle$ and $\langle\alpha e_3+e_4, e_1,e_2,e_6\rangle$ are not invariant under $Y$ for $i=4,5,6$, $j=5,6$ and all $\alpha\in\mathbb{C}$.\\
The subspace  $\langle e_1,e_4,e_5,e_6\rangle$ is not invariant under $Y$ because the third component of $Y(e_4)$ is not zero.\\
Since the third component of $Y(e_2)$ is not zero, it follows that the subspaces $\langle e_1,e_2,e_5,e_6\rangle$ and $\langle e_2,e_4,e_5,e_6\rangle$ are not invariant under $Y$.\\
Assume that the subspace $\langle e_1,e_3,e_5,e_6\rangle$ is invariant under $Y$ then, $Y(e_1)\in\langle e_1,e_3,e_5,e_6\rangle$. So, the second and fourth components of $Y(e_1)$ are zeros. Hence, $\lambda_1^3=q^2$ and $\lambda_1=-q^2$. Thus, $-q^6=q^2$. But, this contradicts the fact that $q$ is a third root of unity. Thus, $\langle e_1,e_3,e_5,e_6\rangle$ is not invariant under $Y$.\\
Note that the sixth and fifth components of $Y(e_4)$ are not zeros. So, the subspaces $\langle e_2,e_3,e_4,e_r\rangle$ are not invariant under $Y$ for $r=5,6$.\\
Since $\lambda_1\neq-1$ it follows that the fourth component of $K_2$ is not zero and $Y(e_2)\not\in\langle e_2,e_3,e_5,e_6\rangle$. Therefore, the subspace $\langle e_2,e_3,e_5,e_6\rangle$ is not invariant under $Y$.\\
Suppose that the subspace $\langle e_3,e_4,e_5,e_6\rangle$ is invariant under $Y$. Then the first and second components of $Y(e_3)$ are zeros. This implies that 
$$\left\{\begin{array}{c}2+q-(1+2q)\lambda_1^2=0\\
1-\lambda_1^3q^2=0
\end{array}
\right.,$$
Thus, $\lambda_1^2=-q$ and $\lambda_1^3=q$ but, this contradicts the fact that $q$ is a primitive third root of unity. Therefore, $\langle e_3,e_4,e_5,e_6\rangle$ is not invariant under $Y$.\\
Consider the subspace $S=\langle\alpha e_3+e_4, e_1,e_5,e_6\rangle$, where $\alpha\in\mathbb{C}$. Suppose that $S$ is invariant under $Y$, then $Y(e_1)\in S$. Then, the second component of $K_1$ is zero. So $\lambda_1^3=q^2$. Conversely, if $\lambda_1^3=q^2$, then the second component of $K_1$ is zero and 
$$\frac{\text{The third component of }K_1}{\text{The fourth component of }K_1}=\frac{-\lambda_1 + q}{-\lambda_1^2 q^2 + 1}=\frac{(-\lambda_1 + q)\lambda_1}{-\lambda_1^3q^2+\lambda_1}=\frac{(-\lambda_1 + q)\lambda_1}{-\-q+\lambda_1}=-\lambda_1.$$
Thus, $Y(e_1)\in S$ and $\alpha=-\lambda_1$. Also, by direct computation and using Mathematica, we show that the second component of each of $Y(e_5)$, $Y(e_6)$ and $Y(-\lambda_1 e_3+e_4)$ is zero. As well as we show that the ratio of the third component to the fourth one of each of these vectors is $-\lambda_1$. This means that $Y(e_5)$, $Y(e_6)$ and $Y(-\lambda_1 e_3+e_4)$ are in $S$.
Therefore, $S$ is invariant under $Y$ if and only if $\lambda_1^3=q^2$ and in this case $\alpha=-\lambda_1$
 and $S=\langle-\lambda_1e_3+e_4,e_1,e_5,e_6\rangle$.\\
It remains to prove that the subspace $S'=\langle\alpha e_3+e_4,e_2,e_5,e_6\rangle$ is not invariant under $Y$ for all $\alpha\in\mathbb{C}$.
Suppose that $S'=\langle\alpha e_3+e_4,e_2,e_5,e_6\rangle$ is invariant under $Y$ for some $\alpha\in\mathbb{C}$, then $Y(e_2)\in S'$.
Then the first component of $K_2$ is zero. Thus, $\lambda_1^2=\frac{2+q}{1+2q}=-q$ (as $q$ is a third root of unity).
Substitute the obtained value of $\lambda_1^2$ in the numerator of the first component of $K_5$ to get
$$ \lambda_1(-2q-1+(2+q)(-q)+(2+q^2)(-q)\lambda_1+(2q+1)q^2\lambda_1)=-3q\lambda_1(1+\lambda_1)\neq0.$$
Hence, $Y(e_5)=K_5\not\in S'$, a contradiction.
\end{proof}
\begin{prop}
If $\lambda_1^3\neq q^2$ then the representation $\rho$ has no  invariant subspaces of dimension 5.
\end{prop}
\begin{proof}
The possible subspaces of dimension 5 that are invariant under $X$ are: $S_6=\langle e_1,e_2,e_3,e_4,e_5\rangle$, $S_5=\langle e_1,e_2,e_3,e_4,e_6\rangle$, $S_4=\langle e_1,e_2,e_3,e_5,e_6\rangle$,
$S_3=\langle e_1,e_2,e_4,e_5,e_6\rangle$, $S_2=\langle e_1,e_3,e_4,e_5,e_6\rangle$, $S_1=\langle e_2,e_3,e_4,e_5,e_6\rangle$ and $S^{\alpha}=\langle e_1,e_2,\alpha e_3+e_4,e_5,e_6\rangle$ for $\alpha\in\mathbb{C}$.\\
Since the third, fifth and sixth components of $K_1$ are not zeros, it follows that $Y(e_1)\not\in S_i$ for $i=3,5,6$. So, the subspaces $S_i$ are not invariant under $Y$ for $i=3,5,6$.\\
Assume that $S_4$ is invariant under $Y$, then $Y(e_1)\in S_4$ and $Y(e_2)\in S_4$. So the fourth components of $K_1$ and $K_2$ are zeros. So, $\lambda_1=-q^2$ and $\lambda_1=-1$ which is impossible because $q$ is a primitive third root of unity. So, $S_4$ is not invariant under $Y$.\\
Since $\lambda_1^2\neq q^3$, the second component of $K_1$ is not zero. So, $Y(e_1)\not\in S_2$. Thus, $S_2$ is not invariant under $Y$.\\
Assume that $S_1$ is invariant under $Y$, then $Y(e_2)\in S_1$. So, the first component of $K_2$ is zero. Thus, $\lambda_1^2=\frac{2+q}{1+2q}=-q$ (since $q$ is a primitive third root of unity). Substitute in numerator of the first component of $K_5$ to get $\lambda_1(-3q-3q\lambda_1)$, which is not zero for $\lambda_1^2=-q$. So $Y(e_5)\not\in S_1$. Therefore, $S_1$ is not invariant under $Y$. \\
It remains to show that $S^\alpha$ is not invariant under $Y$ for all $\alpha\in\mathbb{C}$.
Assume, for some $\alpha\in\mathbb{C}$, that $S^\alpha$ is invariant then, $Y(e_1)$ and $Y(e_2)$ belong to $S^\alpha$. So,
$$\frac{\text{The third component of }K_1}{\text{The fourth component of }K_1}=\frac{\text{The third component of }K_2}{\text{The fourth component of }K_2}=\alpha.$$ 
This implies that
$$\frac{-\lambda_1+q}{-\lambda_1^2 q^2 + 1}=\frac{\lambda_1 - q}{(1 + \lambda_1) (\lambda_1 - q^2)}.$$
So, $$(q^2-1)(\lambda_1^2+\lambda_1+1)=0.$$
Hence, $\lambda_1=q\text{ or }q^2$, a primitive third root of unity. This is a contradiction because $\lambda_1\neq q\text{ and }\lambda_1\neq q^2$. Therefore, $S^\alpha$ is not invariant under $Y$ for all $\alpha\in\mathbb{C}$.
\end{proof}

\begin{thm}
For $\lambda_1\in\mathbb{C}\setminus\{-1,1,q,q^2\}$, the representation $\rho$ is irreducible if and only if $\lambda_1^2\neq -q$, $\lambda_1^2\neq -q^2$,  $\lambda_1^3\neq q$ and $\lambda_1^3\neq q^2$.
\end{thm}
\begin{proof}
It follows directly from Proposition 3.1, Proposition 3.2, Proposition 3.3, Proposition 3.4 and Proposition 3.5.
\end{proof}
\section{Reducible Representation}
By Theorem 3.6, the representation $\rho$ is reducible if and only if $\lambda_1^2=-q$, $\lambda_1^2=-q^2$, $\lambda_1^3=q$ or $\lambda_1^3=q^2$. For $\lambda_1^3=q^2$, $\rho$ is reducible and  the subspace $V$, that is generated by the vectors $-\lambda_1e_3+e_4,\;e_1,\;e_5,\;e_6$, is invariant under $\rho$. Let us write the matrices  representing $\rho(\sigma_1)$ and $\rho(\sigma_2)$ relative to the basis $\{-\lambda_1e_3+e_4,e_1,e_5,e_6,e_4,e_2\}$. 
Let $A=(-\lambda_1e_3+e_4,e_1,e_5,e_6,e_4,e_2)$ be the transition matrix. Then,
$$A=\begin{pmatrix}
0&1&0&0&0&0\\
0&0&0&0&0&1\\
-\lambda_1&0&0&0&0&0\\
1&0&0&0&1&0\\
0&0&1&0&0&0\\
0&0&0&1&0&0\\
\end{pmatrix}
$$
The matrix representing $\sigma_1$ in this basis is
$$A^{-1}XA=\begin{pmatrix}
q^2&0&0&0&0&0\\
0&q&0&0&0&0\\
0&0&\lambda_1^{-1}&0&0&0\\
0&0&0&1&0&0\\
0&0&0&0&q^2&0\\
0&0&0&0&0&\lambda_1\\
\end{pmatrix}
$$
and the matrix representing $\sigma_2$ in the same basis is 
$$A^{-1}YA=\left(C_1\;C_2\;C_3\;C_4\;C_5\;C_6\right),$$
where $$C_1=\begin{pmatrix}
\frac{(1+2q^{1/3}+2q^{2/3} + 2 q +2 q^{4/3})}{3 q (1 + q^{1/3} + q^{2/3} + q)} \\
-\frac{q (2 + q^{2/3}+q +2 q^{5/3}+ 3 q^{7/3})}{3 (-1 + q^{5/3})}\\
\frac{q+q^{5/3} + q^{7/3}}{(-1 + q^{1/3})^3 (1 + q^{1/3})^2 (1 + q^{2/3}) (-1+q^{5/3})}\\
-\frac{q (1 - q^{1/3} + q)}{3 (-1 + q^{1/3})^3 (1 + q^{1/3})^2}\\
0\\
0
\end{pmatrix},\;
C_2=\begin{pmatrix}
\frac{1}{3 q^{4/3} (-1 + q^{4/3})} \\
\frac{q}{-1 + q^{1/3} +q + q^{5/3} - q^{7/3} - q^{2/3}}\\
\frac{q^{1/3}}{(-1 + q^{2/3})^2 (1 + q^{2/3}) (q^{2/3} - q) (-1 + q^{5/3})}\\
\frac{1}{3 (-1 + q^{2/3})^2 (-1 + q) q}\\
0\\
0
\end{pmatrix},$$

$$C_3=\begin{pmatrix}
\frac{(1+2q)(1-q^{4/3})+(2+q)q^{2/3}}{3(1+q)(-q^{2/3}+q^2)} \\
\frac{q^2(1+q^{2/3})}{(q^{2/3}-q)(-1 + q^{5/3})}\\
0\\
-\frac{q^{2/3} (1 + 2 q^2)}{3 (-1 + q^{2/3})^2 (-1 + q)}\\
0\\
0
\end{pmatrix},\;
C_4=\begin{pmatrix}
 -\frac{2 + q + (q-1)(q^{4/3} + 2q^{5/3})}{(q^2-1)q^{5/3}(-q^{2/3}+q^2)}\\
\frac{3(q^{4/3}- q^{8/3} + q)}{-1 + q^{1/3} + q^{5/3} - q^{7/3} - q^{2/3} + q}\\
\frac{(-1 + q) q^{7/3} (q^{4/3} (2 + q)-q^2(1+2 q))}{(-1+q^{2/3})^2 (1 + q^{2/3}) (q^{2/3} -q) (-1 + q^{5/3})}\\
\frac{q^{2/3} - q^{8/3} - (2 + q) + q^{4/3} (-1 + q^2)}{3(-1 + q^{2/3})^2}\\
0\\
0
\end{pmatrix},$$

$$C_5=\begin{pmatrix}
\frac{2 q^{2/3}}{3 (-1 + q^{4/3}}\\
-\frac{q (2 + 2 q^{1/3} + 5 q^{2/3} + 5 q + 2 q^{4/3} + 2 q^{5/3}))}{3 (-1 + q^{5/3}))}\\
 -\frac{q^{7/3}}{(-1 + q^{2/3})^2 (1 + q^{2/3}) (q^{2/3} - q) (-1 + q^{5/3}))}\\
 -\frac{q}{3(-1 + q^{2/3})^2 (-1 + q))}\\
 \frac{q^{2/3}}{1-q^{4/3}}\\
 \frac{1 + q^{2/3} + q^{4/3}}{1 - q^{1/3} + q^{2/3} -q}
\end{pmatrix} \text{ and }
C_6=\begin{pmatrix}
\frac{1}{3 q^{4/3} (-1 + q^{4/3})}\\
\frac{q (2 + q - q^{4/3} (1 + 2 q))}{3 (q^{2/3} - q) (-1 + q^{5/3})}\\
 \frac{q}{(-1 + q^{1/3})^2 (q^{4/3}-1)(1-q^{5/3})}\\
\frac{1}{3 (-1 + q^{2/3})^2 (-1 + q)q}\\
-\frac{q + q^{4/3}+2q^2 - q^{1/3}}{3 - 3 q^{4/3}}\\
-\frac{1 + q^{4/3} + q^{8/3}}{(-1 + q^{1/3})^3 (1 + q^{2/3}) (q + q^{4/3})^2}

\end{pmatrix}.
$$

The restriction $\rho_V$ of $\rho$ to the subspace $V$ is given by:
$\sigma_1\mapsto X'=\begin{pmatrix}
q^2&0&0&0\\
0&q&0&0\\
0&0&\lambda_1^{-1}&0\\
0&0&0&1\\
\end{pmatrix}
$ and 
$\sigma_2\mapsto Y'=\left(E_1\;E_2\;E_3\;E_4\right)$, where
$$E_1=\begin{pmatrix}
\frac{1+2q^{1/3}+2q^{2/3} + 2 q +2 q^{4/3}}{3 q (1 + q^{1/3} + q^{2/3} + q)} \\
-\frac{q (2 + q^{2/3}+q +2 q^{5/3}+ 3 q^{7/3})}{3 (-1 + q^{5/3})}\\
\frac{q+q^{5/3} + q^{7/3}}{(-1 + q^{1/3})^3 (1 + q^{1/3})^2 (1 + q^{2/3}) (-1+q^{5/3})}\\
-\frac{q (1 - q^{1/3} + q)}{3 (-1 + q^{1/3})^3 (1 + q^{1/3})^2}
\end{pmatrix},\;
E_2=\begin{pmatrix}
\frac{1}{3 q^{4/3} (-1 + q^{4/3})} \\
\frac{q}{-1 + q^{1/3} +q + q^{5/3} - q^{7/3} - q^{2/3}}\\
\frac{q^{1/3}}{(-1 + q^{2/3})^2 (1 + q^{2/3}) (q^{2/3} - q) (-1 + q^{5/3})}\\
\frac{1}{3 (-1 + q^{2/3})^2 (-1+q)q}
\end{pmatrix},
$$

$$E_3=\begin{pmatrix}
\frac{(1+2q)(1-q^{4/3})+(2+q)q^{2/3}}{3(1+q)(-q^{2/3}+q^2)}\\
\frac{q^2(1+q^{2/3})}{(q^{2/3}-q)(-1 + q^{5/3})}\\
0\\
-\frac{q^{2/3} (1 + 2 q^2)}{3 (-1 + q^{2/3})^2 (-1 + q)}
\end{pmatrix}\text{ and }
E_4=\begin{pmatrix}
-\frac{2 + q + (q-1)(q^{4/3} + 2q^{5/3})}{(q^2-1)q^{5/3}(-q^{2/3}+q^2)}\\
\frac{3(q^{4/3}- q^{8/3} + q)}{-1 + q^{1/3} + q^{5/3} - q^{7/3} - q^{2/3} + q}\\
\frac{(-1 + q) q^{7/3} (q^{4/3} (2 + q)-q^2(1+2 q))}{(-1+q^{2/3})^2 (1 + q^{2/3}) (q^{2/3} -q) (-1 + q^{5/3})}\\
\frac{q^{2/3} - q^{8/3} - (2 + q) + q^{4/3} (-1 + q^2)}{3(-1 + q^{2/3})^2}
\end{pmatrix}.$$
\begin{thm}
The representation $\rho_V$ is an irreducible representation of $B_3$ of dimension 4.
\end{thm}
\begin{proof}
The vectors $f_1=(1,0,0,0)^T$, $f_2=(0,1,0,0)^T$, $f_3=(0,0,1,0)^T$ and $f_4=(0,0,0,1)^T$ are the eigenvectors of $\sigma_1$ corresponding to the eigenvalues $q^2$, $q$, $\lambda_1^{-1}$ and 1 respectively. Note that each component of $Y'(f_i)=E_i$ is not zero for all $i=1,2,3,4$ except the third component of $E_3$. This implies that all the proper subspaces spanned by any subset of $\{f_1,f_2,f_3,f_4\}$ are not invariant under $Y'$. Hence, $\rho_V$ has no invariant proper subspaces. So, $\rho_V$ is irreducible.
\end{proof}
\section{Conclusion}
We consider a family of representations of $B_3$ constructed by $q$-Pascal's triangle \cite{Al}. We then specialize the parameters used in defining these
representations to non-zero complex numbers.
Kosyak mentioned in \cite{Ko} that the irreducibility of these representations is an open problem for dimensions $\geq6$.
We determine a necessary and sufficient condition for the irreducibility of representations when their dimension is 6.
In addition, we present an irreducible representation of the braid group $B_3$ of dimension 4;
a representation obtained by reducing one of those reducible representations presented in this work.

\bigskip


\end{document}